\documentclass{amsart}
\usepackage{amsmath,amsfonts,amssymb,latexsym}
\usepackage{graphicx}
\usepackage{fancyhdr}
\usepackage{multirow}

\usepackage{cite}

\usepackage{hyperref}

\hypersetup{%
pdftitle={},
pdfsubject={Mathematics},
pdfauthor={},
pdfkeywords={},
hyperindex=true,plainpages=false}

\newcommand{\N}{{\mathbb N}}
\newcommand{\Z}{{\mathbb Z}}

\DeclareMathOperator{\Arg}{Arg}

\theoremstyle{plain}
\numberwithin{equation}{section}
\newtheorem{thm}{Theorem}[section]
\newtheorem{theorem}[thm]{Theorem}

\newtheorem{lemma}[thm]{Lemma}

\newtheorem{conjecture}[thm]{Conjecture}
\newtheorem{question}[thm]{Question}

\theoremstyle{remark}
\newtheorem{remark}{Remark}

\begin{document}

\title[Infinite family of bases]
{An infinite family of multiplicatively independent bases of number systems in cyclotomic number fields}

\author[M. G. Madritsch]{Manfred G. Madritsch}
\address{Manfred Madritsch\newline
\noindent 1. Universit\'e de Lorraine, Institut Elie Cartan de Lorraine, UMR 7502, Vandoeuvre-l\`es-Nancy, F-54506, France;\newline
\noindent 2. CNRS, Institut Elie Cartan de Lorraine, UMR 7502, Vandoeuvre-l\`es-Nancy, F-54506, France}
\email{manfred.madritsch@univ-lorraine.fr}

\author{Volker Ziegler}
\address{Volker Ziegler\newline
\noindent Johann Radon Institute for Computational and Applied Mathematics (RICAM)\newline
\noindent Austrian Academy of Sciences\newline
\noindent Altenbergerstr. 69\newline
\noindent A-4040 Linz, Austria}
\email{volker.ziegler\char'100ricam.oeaw.ac.at}
\thanks{The second author was supported by the Austrian Science Fund (FWF) under the project P~24801-N26.}

\begin{abstract}
  Let $\zeta_k$ be a $k$-th primitive root of unity, $m\geq\phi(k)+1$ an
  integer and $\Phi_k(X)\in\Z[X]$ the $k$-th cyclotomic polynomial.
  In this paper we show that the pair $(-m+\zeta_k,\mathcal N)$ is a
  canonical number system, with $\mathcal
  N=\{0,1,\dots,|\Phi_k(m)|-1\}$.  Moreover we also discuss whether the
  two bases $-m+\zeta_k$ and $-n+\zeta_k$ are multiplicatively
  independent for positive integers $m$ and $n$ and $k$ fixed.
\end{abstract}


\maketitle

\section{Introduction}
Let $q\geq2$ be a positive integer. Then the pair
$(q,\{0,1,\ldots,q-1\})$ is a number system in the positive integers,
\textit{i.e.} every integer $n\in\N$ has a unique representation of
the form
$$n=\sum_{j\geq0}a_jq^j\quad(a_j\in\{0,1,\ldots,q-1\}).$$

Taking a negative integer $q\leq-2$ and $\{0,1,\ldots,\lvert
q\rvert-1\}$ as set of digits one gets a number system for all
integers $\Z$.  This concept of a number system was further extended
to the Gaussian integers by
Knuth~\cite{knuth1981:art_computer_programming}. He showed that the
pairs $(-1+i,\{0,1\})$ and $(-1-i,\{0,1\})$ are number systems in the
Gaussian integers, \textit{i.e.} every Gaussian integer $n\in\Z[i]$
has a unique representation of the form
$$n=\sum_{j\geq0}a_jb^j\quad(a_j\in\{0,1\}),$$
where $b=-1+i$ or $b=-1-i$. It was shown by K{\'a}tai and Szab{\'o}
\cite{katai_szabo1975:canonical_number_systems} that all possible
bases for the Gaussian integers are of the form $-m\pm i$ with $m$ a positive
integer. These results were extended independently by K\'atai and
Kov\'acs \cite{katai_kovacs1981:canonical_number_systems,
  katai_kovacs1980:kanonische_zahlensysteme_in} and
Gilbert~\cite{gilbert81:_radix}, who classified all quadratic
extensions.

A different point of view on the above systems is the following. Let
$P=p_dX^d+p_{d-1}X^{d-1}+\cdots+p_1X+p_0\in\Z[X]$ with $p_d=1$ and
$\mathcal{R}=\Z[X]/P(X)\Z[X]$. Then we call the pair $(P,\mathcal{N})$
with $\mathcal{N}=\{0,1,\ldots,\lvert p_0\rvert-1\}$ a canonical
number system if every element $\gamma\in\mathcal{R}$ has a unique
representation of the form
\[\gamma=\sum_{j\geq0}a_jX^j\quad(a_j\in\mathcal{N}).\]
If $P$ is irreducible and $\beta$ is one of its roots, then
$\mathcal{R}$ is isomorphic to $\Z[\beta]$. In this case we simply
write $(\beta,\mathcal{N})$ instead of $(P,\mathcal{N})$. By setting
$P=X+q$ or $P=X^2+2mX+(m^2+1)$ we obtain the canonical number systems
in the integers and Gaussian integers from above, respectively.

Kov\'acs \cite{kovacs1981:canonical_number_systems} proved that for
any algebraic number field $K$ and order $\mathcal R=\Z[\alpha]$ of
$K$ there exists $\beta$ such that $(\beta,\mathcal{N})$ is a
canonical number system for $\mathcal R$. Moreover he proved that if
$1\leq p_{d-1}\leq \cdots\leq p_1\leq p_0$, $p_0\geq2$ and if $P$ is
irreducible, then $(P,\mathcal{N})$ is a canonical number system in
$\mathcal{R}$. Peth{\H o}
\cite{pethoe1991:polynomial_transformation_and} weakened the
irreducibility condition by only assuming that no root of the
polynomial is a root of unity.

Kov\'acs and Peth{\H o} \cite{kovacs_petho1991:number_systems_in}
provided necessary and sufficient conditions on the pair
$(\beta,\mathcal{N})$ to be a number system in $\Z[\alpha]$. A decade
later Akiyama 
and Peth{\H o} \cite{akiyama_petho2002:canonical_number_systems}
significantly reduced the number of cases one has to check under the
additional assumption that
$$\sum_{i=1}^d\lvert p_i\rvert<p_0.$$
Recently Brunotte \cite{Brunotte:2012} gave a unified proof of the 
results due to Gilbert~\cite{gilbert81:_radix} and Kov\'acs and
Peth{\H o}~\cite{kovacs_petho1991:number_systems_in} on irreducible CNS polynomials
and extended them to reducible polynomials.

Let us denote by $\zeta_k$ some primitive $k$-th root of unity. Since
$+i$ and $-i$ are primitive fourth roots of unity, we can say that all
the bases for the Gaussian integers are of the form $-m+\zeta_4$, with
$m\geq 1$. Furthermore an easy computation using the classification
for quadratic extensions shows that $-m+\zeta_3$ and $-m+\zeta_6$ with
$m\geq1$ are bases in $\Z[\zeta_3]$ and $\Z[\zeta_6]$,
respectively. Thus our first result answers the question, whether
$-m+\zeta_k$ is a basis in $\Z[\zeta_k]$ for general $k>2$.

\begin{theorem}\label{th:bases}
  Let $k>2$ and $m$ be positive integers. If $m\geq\phi(k)+1$, then
  $\{-m+\zeta_k,\mathcal{N}\}$ is a canonical number system.
\end{theorem}

Our second result considers, whether these bases are multiplicatively
independent. We call two algebraic integers $\alpha$ and $\beta$
multiplicatively independent, if $\alpha^p=\beta^q$ has only the
solution $p=q=0$ over the integers. Considering bases in the Gaussian
integers Hansel and Safer \cite{hansel_safer2003:vers_un_theoreme}
proved that $-m+i$ and $-n+i$ are multiplicative independent for all
$m>n>0$. Their motivation was a version of Cobham's theorem
\cite{cobham1972:uniform_tag_sequences} for the Gaussian
integers. Cobham's theorem states that a set $E\subset\N$ is $m$- and
$n$-recognizable with $m$ and $n$ multiplicative independent integers
if and only if the set $E$ is a union of arithmetic progressions. This
means that e.g. except for arithmetic progressions we are not able to
deduce the base $3$ expansion form the one in base $2$.

In the present paper we want to generalize the result of Hansel and
Safer \cite{hansel_safer2003:vers_un_theoreme} on the multiplicative
independence to bases of $\Z[\zeta_k]$ of the form $-m+\zeta_k$ for
$k\geq3$.

\begin{theorem}\label{th:mult_independent}
  Let $k\geq3$ be a positive integer. Then the algebraic integers
  $-m+\zeta_k$ and $-n+\zeta_k$ are multiplicatively independent
  provided $m>n>C(k)$, where $C(k)$ is an effective computable
  constant depending on $k$.

  Moreover, if $k$ is a power of $2,3,5,6,7,11,13$, $17,19$ or $23$,
  then $-m+\zeta_k$ and $-n+\zeta_k$ are multiplicatively independent
  as long as $m>n>0$ if $k\neq 6$ and $m>n>1$ otherwise.
\end{theorem}

In order to prove Theorem~\ref{th:mult_independent} one has to show
that the Diophantine equation
\begin{equation}\label{eq:mult_independent}
(-m+\zeta_k)^p=(-n+\zeta_k)^q
\end{equation}
has no solution $(n,m,p,q)$, with $m>n>C(k)$ and $p,q$ are not both
zero. Taking norms in \eqref{eq:mult_independent} we obtain
$\Phi_k(m)^p=\Phi_k(n)^q$, where $\Phi_k$ denotes the $k$-th
cyclotomic polynomial, and therefore the Diophantine
equation~\eqref{eq:mult_independent} becomes
\begin{equation}\label{eq:Norm2}
\Phi_k(m)^p=\eta^q
\end{equation}
with $\eta=\Phi_k(n)$. In case that $p$ and $q$ have
greatest common divisor $d>1$ we take $d$-th roots and obtain equation
\eqref{eq:Norm2} with $p$ and $q$ coprime. Therefore we may assume
that $p$ and $q$ are coprime and we deduce that
$\eta$ is a $p$-th power, i.e.  we get an equation of the form
$\Phi_k(m)=y^q$ with $y=\eta^{1/p}$.   Note that since we assume
$m,n>C(k)$ the case $\Phi_k(m)=-y^q$ can be excluded. Indeed, we know
that $\Phi_k(m)>0$ for $m$ sufficiently large. In any case we obtain a
Diophantine equation of the form
\begin{equation}\label{eq:Norm}
\Phi_k(m)=y^q.
\end{equation}
Equation \eqref{eq:Norm} is closely related to the well-studied Nagell-Ljunggren equation
\begin{equation}\label{eq:Nagell}
\frac{x^k-1}{x-1}=y^q\qquad x,y>1,\; q\geq 2, \; k>2.
\end{equation}
If $k$ is a prime, then the Diophantine equations \eqref{eq:Norm} and \eqref{eq:Nagell} are identical and every solution to \eqref{eq:mult_independent} yields 
a solution to \eqref{eq:Nagell}. Therefore we conjecture: 

\begin{conjecture}
Let $k$ be an odd prime. Then $-m+\zeta_k$ and $-n+\zeta_k$ are multiplicatively independent provided $m>n>0$.
\end{conjecture}

Since it is widely believed that the Nagell-Ljunggren equation \eqref{eq:Nagell} has only the three solutions
\[\frac{3^5-1}{3-1}=11^2,\qquad \frac{7^4-1}{7-1}=20^2, \qquad \frac{18^3-1}{18-1}=7^3\]
this conjecture seems plausible. Note that none of these solutions to \eqref{eq:Nagell} implies a solution to \eqref{eq:mult_independent}. We guess that even 
more 
is true and believe that the following question has an affirmative answer.

\begin{question}
 Given $k>2$ are $-m+\zeta_k$ and $-n+\zeta_k$ multiplicatively independent provided $m>n>1$?
\end{question}

We excluded the case that $n=1$ since for $k=6$ we have $-1+\zeta_6=\zeta_3$ and $\zeta_3$ and $-m+\zeta_6$ are always multiplicatively 
dependent.

\section{Canonical number systems and the proof of Theorem \ref{th:bases}}

The main tool in our proof is the following result due to Peth{\H o}
\cite{pethoe1991:polynomial_transformation_and}.

\begin{lemma}[Peth\H{o} {\cite[Theorem 7.1]{pethoe1991:polynomial_transformation_and}}]\label{lem:pethoe}
  Let $P=\sum_{j=0}^dp_jX^j\in\Z[X]$ with $p_d=1$. If $0<p_{d-1}\leq
  p_{d-2}\leq\cdots\leq p_0$, $p_0\geq2$ and no root of $P$ is a root
  of unity, then the pair $(P,\mathcal{N})$ is a canonical number
  system.
\end{lemma}

Since $\zeta_k$ is a primitive $k$-th root of unity $-m+\zeta_k$ is a
root of $P(X)=\Phi_k(m+X)=\sum_{d=0}^{\phi(k)}p_dX^d$. We have
$p_{\phi(k)}=1$ and for $d=1,\ldots,\phi(k)$, by Vieta's
formula we obtain
\begin{equation}\label{eq:coeff} p_{\phi(k)-d}=\tilde
  p_d=\sum_{a_1<\cdots<a_d\in(\Z/k\Z)^*}(m-\zeta_k^{a_1})\cdots(m-\zeta_k^{a_d}).
\end{equation}
Here we identify elements of $\Z/k\Z$ with their minimal
representatives in $\Z$, i.e. with the integers in the set
$\{0,1,\dots, k-1\}$ in the obvious way. 

Before checking the requirements of Lemma \ref{lem:pethoe} we provide
general lower and upper bounds for $\tilde p_{d}$. Since
$$\left|\Arg(m-\zeta_k^a)\right|\leq \arctan\left(\frac 1{m-1}\right)\leq\frac1{m-1}$$
we have that
\[\left|\Arg\left((m-\zeta_k^{a_1})\cdots(m-\zeta_k^{a_d})\right)\right|
\leq\frac d{m-1}\leq 1.\]

Therefore we have
$$\mathrm{Re}\left((m-\zeta_k^{a_1})\cdots(m-\zeta_k^{a_d})\right)\geq (m-1)^d \cos(1)\geq \tfrac12 (m-1)^d,$$
which together with \eqref{eq:coeff} yields the bounds
\begin{equation}\label{eq:bounds}\tfrac12(m-1)^d\binom{\phi(k)}{d}\leq\tilde
  p_d\leq(m+1)^d\binom{\phi(k)}{d}.\end{equation}

Now we use these bounds in order to check the requirements of Lemma
\ref{lem:pethoe}. We start by showing that $p_0=\tilde p_{\phi(k)}\geq
2$. Using the lower bound of \eqref{eq:bounds} we get that
\[p_0=\tilde p_{\phi(k)}\geq\frac{(m-1)^{\phi(k)}}2\geq2\]
provided that $k,m\geq3$ which we supposed.

In the next step we recursively prove that $\tilde p_d\leq\tilde
p_{d+1}$ for $d=1,\ldots,\phi(k)-1$. For a given ordered $(d+1)$-tuple
$a_1<\cdots<a_{d+1}$ there are $d+1$ possibilities to deduce an
ordered $d$-tuple. Thus
\begin{multline*}(d+1)\sum_{a_1<\cdots<a_{d+1}\in(\Z/k\Z)^*}(m-\zeta_k^{a_1})\cdots(m-\zeta_k^{a_{d+1}})\\
=\sum_{a_1<\cdots<a_{d}\in(\Z/k\Z)^*}(m-\zeta_k^{a_1})\cdots(m-\zeta_k^{a_{d}})\sum_{\substack{b\in(\Z/k\Z)^*\\b\neq 
a_1,\ldots,a_d}}(m-\zeta_k^b)
\end{multline*}

Therefore we get for $d<\phi(k)$ that
\begin{align*}
\tilde p_{d+1}
=&\sum_{a_1<\cdots<a_{d+1}\in(\Z/k\Z)^*}(m-\zeta_k^{a_1})\cdots(m-\zeta_k^{a_{d+1}})\\
=&\frac{1}{d+1}\sum_{a_1<\cdots<a_{d}\in(\Z/k\Z)^*}(m-\zeta_k^{a_1})\cdots(m-\zeta_k^{a_{d}})\sum_{\substack{b\in(\Z/k\Z)^*\\ b\neq 
a_1,\ldots,a_d}}(m-\zeta_k^b)\\
=&\frac{m(\phi(k)-d)}{d+1}\sum_{a_1<\cdots<a_{d}\in(\Z/k\Z)^*}(m-\zeta_k^{a_1})\cdots(m-\zeta_k^{a_{d}})\\
&-\frac{1}{d+1}\sum_{a_1<\cdots<a_{d}\in(\Z/k\Z)^*}(m-\zeta_k^{a_1})\cdots(m-\zeta_k^{a_{d}})\sum_{\substack{b\in(\Z/k\Z)^*\\ b\neq 
a_1,\ldots,a_d}}\zeta_k^b\\
\geq& \frac{\tilde p_dm(\phi(k)-d)}{d+1}\\
&\quad-\frac{1}{d+1}\left|\sum_{a_1<\cdots<a_{d}\in(\Z/k\Z)^*}(m-\zeta_k^{a_1})\cdots(m-\zeta_k^{a_{d}})\sum_{\substack{b\in(\Z/k\Z)^*\\
b\neq a_1,\ldots,a_d}}\zeta_k^b \right|\\
\geq & \frac{\tilde p_dm(\phi(k)-d)}{d+1}-\binom{\phi(k)}{d}\frac{(m+1)^d(\phi(k)-d)}{d+1}.
\end{align*}
Note that in the first four lines in the inequality above all the sums
over complex terms yield indeed real numbers. Further the ``$\geq$''
signs are true due to the triangle inequality.

By our lower bound in \eqref{eq:bounds} we get that
\[\binom{\phi(k)}{d}\frac{(m+1)^d}{\tilde p_d}\leq2\left(\frac{m+1}{m-1}\right)^d=2\left(1+\frac{2}{m-1}\right)^d<2e^2,\]
where we have used that $m-1\geq\phi(k)>d$ by our assumption. Altogether we
obtain the inequality
\[\tilde p_{d+1}> \frac{\tilde p_d(m-2e^2)(\phi(k)-d)}{d+1},\]
which implies that the coefficients are increasing if 
\begin{equation}\label{eq:bound-m-d}
\frac{(m-2e^2)(\phi(k)-d)}{d+1}\geq1.
\end{equation}

We split our considerations into two cases whether $d\leq\phi(k)-2$ or not.
\begin{itemize}
\item\textbf{Case 1: $d\leq\phi(k)-2$.} Together with our assumption
  that $m\geq\phi(k)+1$ we get that the coefficients $\tilde p_d$ up
  to $\tilde p_{\phi(k)-1}$ are increasing
  if $$\frac{2(\phi(k)+1-2e^2)}{\phi(k)-1}\geq1.$$ Since $\phi(k)$ only
  takes positive values a simple calculation yields that this is the
  case for $\phi(k)\geq 4e^2-3\approx 26.56$.

  On the contrary, we may assume $\phi(k)\leq 26$, which implies
  $d\leq24$. Plugging this into \eqref{eq:bound-m-d} yields that the
  coefficients are increasing provided $m\geq e^2+\frac{25}2\approx
  19.89$.

  Finally computing all polynomials $\Phi_k(m+x)$ with $\phi(k)+1\leq
  m\leq 19$ by aid of a computer algebra system we see that in any
  case the coefficients $\tilde p_d$ are increasing. In fact we only
  have to check 300 polynomials, which are certainly to many for
  writing them down.
\item\textbf{Case 2: $d=\phi(k)-1$.} Thus we are left to show that
  $p_1\leq p_0$, i.e.
\[\prod_{a \in (\Z/k\Z)^*}(m-\zeta_k^a) \geq \sum_{b \in (\Z/k\Z)^*}\prod_{\substack{a \in (\Z/k\Z)^*\\ a\neq b}}(m-\zeta_k^a).\]
Since
\begin{align*}
1&\geq \frac{\phi(k)}{m-1}
  =\sum_{b \in (\Z/k\Z)^*}\frac{1}{m-1}
  \geq\sum_{b \in (\Z/k\Z)^*}\frac{1}{\left| m-\zeta_k^b\right|}\\
&\geq\left|\sum_{b \in (\Z/k\Z)^*}\frac{1}{m-\zeta_k^b}\right|
=\sum_{b \in (\Z/k\Z)^*}\frac{1}{m-\zeta_k^b}=\frac{p_1}{p_0}
\end{align*}
provided that $m\geq\phi(k)+1$, this case is settled too.
\end{itemize}

Finally the last requirement, in order to apply Lemma
\ref{lem:pethoe}, is to show that no root of $\Phi_k(m+X)$ is a root
of unity. Since we assume $m\geq \phi(k)+1\geq 3$. So all roots of
$\Phi_k(m+X)$ have real part $\geq 2$ and therefore cannot be roots of
unity.

Therefore $\Phi_k(X+m)$ satisfies all assumptions of Lemma
\ref{lem:pethoe} and Theorem \ref{th:bases} is proved.

\begin{remark}
  In case that $m=\phi(k)$ both situations $p_1\leq p_0$ and $p_1>p_0$
  can occur. For instance, in the case $k=11$ and $m=\phi(k)=10$ we
  obtain $p_1<p_0$ and in the case $k=22$ and $m=\phi(k)=10$ we obtain
  $p_1>p_0$. Therefore in view of the method of proof Theorem
  \ref{th:bases} is best possible.
\end{remark}

\section{Multiplicative independence}

Before we start with the proof of Theorem \ref{th:mult_independent}, let us note that if $p$ and $q$ have a
common factor $d$, then by taking $d$-th roots on both sides of equation \eqref{eq:mult_independent} we obtain
\begin{equation}\label{eq:mult_ind2}
 \zeta_k^j(-m+\zeta_k)^p=(-n+\zeta_k)^q,
\end{equation}
with  $\gcd(p,q)=1$ and $0\leq j<k$.

We start with the proof of the first statement of Theorem
\ref{th:mult_independent}. In view of the discussion below Theorem \ref{th:bases} we consider the Diophantine
equation
\begin{equation}\label{eq:Supper_elliptic} f(x)=y^q,\end{equation}
where $f$ is a polynomial with rational coefficients and with at least
two distinct roots. Obviously \eqref{eq:Norm} fulfills these
requirements as soon as $k>2$.  Due to Schinzel and Tijdeman
\cite{Schinzel:1976} (see also \cite[Theorem 10.2]{Shorey:Dioph}) we
know that for all solutions $(x,y,q)$ to \eqref{eq:Supper_elliptic}
with $|y|>1$, $q$ is bounded by an effectively computable constant
depending only on $f$. In view of Theorem \ref{th:mult_independent} we
see that all solutions coming from~\eqref{eq:mult_independent} yield
$y>1$ provided $n\neq 0$. But it is also well known (see
e.g. \cite[Theorem 6.1 and 6.2]{Shorey:Dioph}) that for given $q\geq
2$ all solutions $(x,y)$ to \eqref{eq:Supper_elliptic} satisfy
$\max\{|x|,|y|\}<C_{f,q}$, where $C_{f,q}$ is an absolute, effectively
computable constant depending on $f$ and $q$. Thus we conclude that
$q=1$ in equation \eqref{eq:mult_ind2} if $m,n>C(k)$, and $C(k)$ is an
effectively computable constant only depending on $k$. Exchanging the
roles of $m$ and $n$ we deduce that $q=p=1$ if $m,n>C(k)$. But this
yields the equation $\zeta_k^j(-m+\zeta_k)=(-n+\zeta_k)$ and taking
norms on both sides gives $\Phi_k(m)=\Phi_k(n)$. Since $\Phi_k(m)$ is
strictly increasing for $m$ large enough we deduce that $m=n$ which we
excluded. Therefore the proof of the first part of
Theorem~\ref{th:mult_independent} is complete.

However, computing the constant $C(k)$ using the methods cited above
will yield very huge constants even for small values of $k$. Therefore
we will use a more direct approach in the proof of the second part.

We note that if $q$ is the product of all different prime divisors of
$n$ (its radical), then $\Phi_n(X)=\Phi_{q}(X^{n/q})$. In
particular, if $n=q^\ell$, then we have
\begin{equation}\label{eq:PowerRed}
\Phi_{q^{\ell}}(X)=\Phi_{q}\left(X^{q^{\ell -1}}\right). 
\end{equation}
Let us start with powers of $2$. In this case we have that
$\Phi_{2^\ell}(X)=X^{2^{\ell -1}}+1$. Therefore the case that $k$ is a
power of $2$ is deduced from the fact that the Diophantine equation
$$X^2+1=Y^q$$
has no solution. This is a special case of Catalan's equation
$X^p-Y^q=1$ which was completely solved by Mih{\v{a}}ilescu
\cite{Mihailescu:2004}. Note that this special case was already proved
by Chao \cite{Chao:1965}.

In the next step we consider the case that $k$ is a power of $3$ or $6$.

\begin{lemma}\label{lem:k=3}
Let $k$ be a power of $3$ or $6$. Then the two algebraic integers $-m+\zeta_k$ and 
$-n+\zeta_k$ are multiplicatively independent provided $n\neq m$ and $n,m\not\in\{-1,0,1\}$.
\end{lemma}

\begin{proof}
  Let us consider equation \eqref{eq:Norm} and let $M=m^{k/3}$ if $k$
  is a power of $3$ and $M=m^{k/6}$ if $k$ is a power of $6$. Then
  equation \eqref{eq:Norm} turns into
  \[M^2\pm M+1=\left(\frac{2M\pm 1}2\right)^2+\frac 34=Y^q,\] where
  the ``$+$'' sign holds in case that $k$ is a power of $3$ and the
  ``$-$'' sign holds in case that $k$ is a power of $6$.  Multiplying
  this equation by $4$ and putting $X=2M\pm 1$ we get
  \[X^2+3=4Y^q.\] Due to Luca, Tengely and Togbe\cite{Luca:2009} we know
  that $(X,Y,q)=(1,1,q)$ and $(37,7,3)$ are the only solutions with
  $X,Y,q\geq 1$. The first solution yields $m\in\{-1,0,1\}$ which we
  have excluded and the second solution yields $M=\pm 18$ or $M=\pm
  19$. Since both, $18$ and $19$, are not perfect powers we conclude
  that $k=3$ or $k=6$.  We may recover $m$ and $n$ from the
  substitutions that led to equation \eqref{eq:Norm}. Therefore we are
  left to check that in case that $k=3$ equation
  \eqref{eq:mult_independent} has no solution with $m=18,-19$,
  $n=2,-3$, $q=3$ and $p=1$ and in case that $k=6$ equation
  \eqref{eq:mult_independent} has no solution with $m=-18,19$,
  $n=-2,3$, $q=3$ and $p=1$. This is of course easily done e.g. by a
  computer.
\end{proof}

We are left to the case that $k$ is a power of $5,7,11,13,17,19$ or
$23$. Since every solution to \eqref{eq:mult_independent} implies a
solution to the Nagell-Ljunggren equation \eqref{eq:Nagell}, the
following results due to Bugeaud, Hanrot and Mignotte \cite[Theorem 1
and 2]{Bugeaud:2002} are essential in the proof of Theorem
\ref{th:mult_independent}.

\begin{theorem}\label{th:Bugeaud}
The Nagell-Ljunggren equation \eqref{eq:Nagell} has except the three solutions $(x,y,k,q)=(3,11,5,2),(7,20,4,2),(18,7,3,3)$ no further solution in case that
 \begin{itemize}
  \item $k$ is a multiple of $5,7,11$ or $13$, or
  \item $k$ is a multiple of $17,19$ or $23$ and $q\neq 17,19$ or $23$ respectively.
 \end{itemize}
\end{theorem}

Using identity \eqref{eq:PowerRed} we get in case that $k=P^\ell$ is a prime power
\[\Phi_{k}(X)=\frac{\left(X^{P^{\ell-1}}\right)^P -1}{\left(X^{P^{\ell-1}}\right)-1}=\frac{x^P-1}{x-1}\]
with $x=X^{P^{\ell -1}}$. If we assume now that $P=5,7,11$ or $13$,
then the only possible solutions to \eqref{eq:mult_ind2} imply that
$q=1$. Exchanging the role of $m$ and $n$ we also obtain $p=1$ and
equation \eqref{eq:mult_ind2} turns into
$$\zeta_k^j(-m+\zeta_k)=-n+\zeta_k.$$
Taking norms yields
$$\frac{m^k-1}{m^{k/P}-1}=\frac{n^k-1}{n^{k/P}-1},$$
hence $m=n$, which we excluded. Therefore we are left with the case
that $k$ is a power of $17,19$ or $23$.

Before we consider the remaining cases we want to remind the reader
that an integral basis for $\Z[\zeta_k]$ is given by
$\{1,\zeta_k,\ldots,\zeta_k^{\phi(k)-1}\}$ and therefore every
algebraic integer $\alpha\in\Z[\zeta_k]$ has a unique representation
of the form
\begin{equation}\label{eq:Rep}
\alpha=\sum_{i=0}^{\phi(k)-1} a_i\zeta_k^i, 
\end{equation}
with $a_i\in\Z$ for $0\leq i<\phi(k)$. In the following we will expand
the left and right side of equation \eqref{eq:mult_ind2} in various
cases and bring the result into the unique form \eqref{eq:Rep}. We will
show below that on the left and right side of \eqref{eq:mult_ind2} the
unique representations~\eqref{eq:Rep} do not match and conclude that
no solution to \eqref{eq:mult_ind2} and hence no solution to
\eqref{eq:mult_independent} exists.

Assume that \eqref{eq:mult_ind2} has a solution $(m,n,p,q,j)$. In case that $k$ is a power of $17,19$ or $23$ Theorem \ref{th:Bugeaud} implies that the only 
prime divisors of $p$ and $q$ in equation \eqref{eq:mult_independent} are $17,19$ 
or $23$ respectively. Therefore we are left to equation \eqref{eq:mult_ind2}, with $p=1$ and $q=17,19$ or $23$. Note that in case of $k=q$ we have
\begin{align*}
(-n+\zeta_k)^k=&\sum_{i=0}^{k} (-n)^{k-i}\binom{k}{i} \zeta_k^{i}\\
=&\sum_{i=0}^{k-2} \left((-n)^{k-i}\binom{k}{i}+kn\right)\zeta_k^i +1=\sum_{i=0}^{k-2} p_{k,i}(n)\zeta_k^i.
\end{align*}

Let us assume that $0\leq j< k-2$. Then with $p=1$ and $k=q$, either the coefficients $p_{k,0}(n)$ and 
$p_{k,1}(n)$ have to 
vanish simultaneously or the coefficients  $p_{k,3}(n)$ and $p_{k,4}(n)$ vanish simultaneously. Since
\[\gcd(p_{17,0}(n),p_{17,1}(n))=\gcd(-n^{17} + 17 n + 1, 17 n^{16} + 17 n)=1 \]
as polynomials in $n$, the coefficients cannot vanish simultaneously. Similarly we obtain
\begin{gather*}
\gcd(p_{19,0}(n),p_{19,1}(n))=\gcd(p_{23,0}(n),p_{23,1}(n))=1\\
\gcd(p_{17,3}(n),p_{17,4}(n))=17n, \quad \gcd(p_{19,3}(n),p_{19,4}(n))=19n,\\
\gcd(p_{23,3}(n),p_{23,4}(n))=23n,
\end{gather*}
where we consider the $\gcd$ taken over the polynomial ring $\Z[n]$. Therefore we deduce $n=0$, a solution which we excluded.

Now consider the case that $k=q$ and $j=k-2$. In this case equation \eqref{eq:mult_ind2} turns into
\[-1-\zeta_k-\cdots-(m+1)\zeta_k^{k-2n}=(-n+\zeta_k)^k=\sum_{i=0}^{k-2} p_{k,i}(n)\zeta_k^i.\]
Considering the coefficient of $\zeta_k$ we have
\[-1=kn(1+n^{k-2})\]
which cannot hold if $n>1$.
Similarly for $j=k-1$ we obtain
\[(m+1)+m\zeta_k+\cdots+m\zeta_k^{k-2}=(-n+\zeta_k)^k=\sum_{i=0}^{k-2} p_{k,i}(n)\zeta_k^i.\]
And in particular the coefficients of $\zeta_k$ and $\zeta_k^2$ have to be equal, i.e.
\[kn+kn^{k-1}=-\frac{k(k-1)}2 n^{k-2}+kn,\]
or equivalently
\[0=kn^{k-2}(2n+k-1).\] The last equation can only hold if $n$ or $k$
is zero or $2n+k-1=0$ which implies $n=-\frac{k-1}2<0$. Since all these
options are excluded this case cannot occur.

Now let us consider the case that $k=q^\ell$, with $\ell>1$, and $q=17,19$ or $23$. If $j<\phi(k)-1$ in \eqref{eq:mult_ind2} then all coefficients $\zeta_k^i$ 
with $0\leq i <\phi(k)$ of $\zeta_k^j(-m+\zeta_k)$ vanish except exactly two. Indeed we have
\[\zeta_k^j(-m+\zeta_k)=-m\zeta_k^j+\zeta_k^{j+1}.\]
On the right side of \eqref{eq:mult_ind2} all the coefficients of $\zeta_k^i$ with $0\leq i \leq q$ do not vanish, i.e. a contradiction.

In case that $j=\phi(k)+r$ for some $0\leq r < k-\phi(k)-1=q^{\ell -1}-1$ we have
\begin{multline*}
(-m+\zeta_k)\zeta_k^j=(-m+\zeta_k)\left(-\zeta_k^r-\zeta_k^{r+q^{\ell-1}}-\zeta_k^{r+2q^{\ell-1}}-\cdots-\zeta_k^{r+(q-2)q^{\ell-1}}\right)\\
=-m\zeta_k^r +\zeta_k^{r+1}-\cdots -m\zeta_k^{r+(q-2)q^{\ell-1}}+\zeta_k^{r+1+(q-2)q^{\ell-1}}
\end{multline*}
i.e. there exist non vanishing coefficients of $\zeta_k^i$ with $i>q$. But on the right side of \eqref{eq:mult_ind2} all coefficients of $\zeta_k^i$ 
vanish for $i>q$ and again no solution is possible.

We are left with the two cases that $j=\phi(k)-1$ and $j=k-1$. In both cases it is easy to see that $\zeta_k^j(-m+\zeta_k)$ has non-vanishing coefficients of 
$\zeta_k^i$ for $i>q$ and we deduce again a contradiction.

\section*{Acknowledgment}

The authors want to thank Shigeki Akiyama (University of Tsukuba) for
pointing them to the refinement of Peth{\H o} \cite[Theorem
7.1]{pethoe1991:polynomial_transformation_and} and Julien Bernat
(Universit\'e de Lorraine) for many valuable discussions concerning
Cobham's theorem and its requirements.

We also want to thank the anonymous referee, who read the manuscript
very carefully and his/her suggestions considerably improve the
presentation of the results.

\bibliographystyle{abbrv}
\bibliography{Bases}

\end{document}